\documentclass{amsproc}

\usepackage[ansinew]{inputenc}
\usepackage{graphicx}
\usepackage{color}

\usepackage{enumerate,latexsym}
\usepackage{latexsym}
\usepackage{amsmath,amssymb}
\usepackage{graphicx}
 
\usepackage{float}
\usepackage{subfig}
\newfont{\bb}{msbm10 at 10pt}
\newfont{\bbsmall}{msbm8 at 8pt}

\newtheorem{teorema}{Theorem}

\newtheorem{lema}{Lemma}
\newtheorem{corolario}{Corollary}
\newtheorem{definicion}{Definition}

\newtheorem{nota}{Remark}
\newtheorem*{ejemplo*}{Examples}
\newtheorem{claim}{Claim}

\def\n{\hbox{\bb N}}

\def\r{\hbox{\bb R}}

\def\fl{\longrightarrow}

\def\c{\hbox{\bb C}}

\def\cA{\mathcal{A}}

\def\cB{\mathcal{B}}

\def\cO{\mathcal{O}}

\def\cV{\mathcal{V}}
\def\cH{\mathcal{H}}

\let\8=\infty \let\0=\emptyset  
\let\hat=\widehat

\let\tilde=\widetilde

\let\alfa=\alpha

\def\flecha{\rightarrow}
\def\esiz{\langle}
\def\esde{\rangle}

\def\cte.{\mathop{\rm cte.}\nolimits}

 \def\Im{\mathop{\rm Im }\nolimits}

\def\N{\mathbb{N}}
\def\L{\mathbb{L}}
\def\A{\mathbb{A}}
\def\Z{\mathbb{Z}}

\def\R{\mathbb{R}}

\def\C{\mathbb{C}}
\def\D{\mathbb{D}}

\def\H{\mathbb{H}}
\def\S{\mathbb{S}}

\newcommand{\beq}{\begin{equation}}
\newcommand{\eeq}{\end{equation}}

\numberwithin{equation}{section}

\begin{document}

\begin{title}[Isolated singularities of the prescribed mean curvature equation]{Isolated singularities of the prescribed mean curvature equation in Minkowski $3$-space}
\end{title}
\today
\author{José A. Gálvez}
\address{José A. Gálvez, Departamento de Geometría y Topología,
Universidad de Granada, 18071 Granada, Spain}
 \email{jagalvez@ugr.es}

\author{Asun Jiménez}
\address{Asun Jiménez, Departamento de Geometria, IME,
Universidade Federal Fluminense, 24.020-140  Niterói, Brazil}
 \email{asunjg@vm.uff.br}

\author{Pablo Mira}
\address{Pablo Mira, Departamento de Matemática Aplicada y Estadística, Universidad Politécnica de
Cartagena, 30203 Cartagena, Murcia, Spain.}

\email{pablo.mira@upct.es}

\thanks{
\textit{Keywords and phrases:} quasilinear elliptic equation, isolated singularity, prescribed mean
curvature, boundary regularity.\\
{\it 2010 Mathematics Subject Classification:} 35J62, 53C42.\\
The authors were partially supported by
MINECO, Grant No. MTM2016- 80313-P, Junta de Andalucía Grant No.
FQM325, Programa de Apoyo a la Investigación, Fundación Seneca-Agencia de Ciencia y Tecnología
Región de Murcia, reference 19461/PI/14  
and Conselho
Nacional de Desenvolvimento Científico e Tecnológico - CNPq - Brasil, reference 405732/2013-9.}



\maketitle

\begin{center}
\begin{abstract}
We give a classification of non-removable isolated singularities for real analytic solutions of the prescribed mean curvature equation in Minkowski $3$-space.
\end{abstract}
\end{center}

\section{Introduction}

In this paper we study non-removable isolated singularities of the following quasilinear, non-uniformly elliptic PDE in two variables:

\begin{equation}\label{maineq}
(1-z_y^2) z_{xx} + 2 z_x z_y z_{xy} + (1-z_x^2) z_{yy} = 2 \cH (1-z_x^2-z_y^2)^{3/2},
\end{equation}
where $\cH=\cH(x,y,z)$ is a $C^k$ positive function ($k\geq 1$) and $z=z(x,y)$ satisfies the ellipticity condition $z_x^2+z_y^2<1$. The solutions of this equation have a geometric interpretation, since they represent spacelike graphs of prescribed mean curvature $\cH$ in the Lorentz-Minkowski space $\L^3$.

More specifically, we will consider elliptic solutions $z(x,y)$ to \eqref{maineq} that are $C^2$ on a punctured disk $$\Omega=\{(x,y): (x-x_0)^2+(y-y_0)^2<\rho^2\}\subset \R^2,$$ and do not extend smoothly to the puncture. The radius $\rho$ of $\Omega$ is irrelevant here, since we will consider two solutions $z_1,z_2$ in the above conditions to be equal if they coincide in a punctured neighborhood of $(x_0,y_0)$. 

Our aim is to describe in detail the asymptotic behavior of elliptic solutions $z(x,y)$ to \eqref{maineq} around such a non-removable isolated singularity, and to classify the associated moduli space.

In \cite{Bar} Bartnik proved that if $z(x,y)$ is an elliptic solution to \eqref{maineq} that presents a non-removable singularity at the origin, then it has a \emph{conelike} behavior: $z$ extends continuously to the origin (say, $z(0,0)=0$), and $$\lim_{(x,y)\to(0,0)} \frac{|z(x,y)|}{\sqrt{x^2+y^2}} =1,$$ i.e. $z(x,y)$ is asymptotic at the origin to the upper or lower \emph{light cone}; see e.g. Figure \ref{f1} for an example of this situation.

The case $\cH=0$ in \eqref{maineq} corresponds to the well-studied situation of maximal surfaces in $\L^3$. A description of the asymptotic behavior of maximal surfaces around non-removable isolated singularities was obtained by Fernández, López and Souam through the Weierstrass representation of such surfaces, see Lemma 2.1 in \cite{FLS2}. For other works about isolated singularities of maximal surfaces, see  \cite{Ec}, \cite{Ko}, \cite{FLS1}, \cite{EsRo}, \cite{KlMi}, \cite{UmYa}.

 In the constant mean curvature case we refer to \cite{Br} and \cite{Um} for previous results concerning singular surfaces. 

Our main result in this paper gives a classification of the elliptic solutions to \eqref{maineq} that have a non-removable isolated singularity at the origin, in the case that $\cH=\cH(x,y,z)$ is real analytic and the curvature of the graph $z=z(x,y)$ in $\L^3$ does not vanish around the singularity:

 \begin{teorema}\label{main1}
  Let $\cO\subset \R^3$ be a neighborhood of some point $p_0=(x_0,y_0,z_0)$, and let $\cH\in C^{\omega}(\cO)$, $\cH>0$.
 
 Let $\cA_1$ denote the class of all elliptic solutions $z(x,y)$ to \eqref{maineq} that satisfy the following conditions:
  \begin{enumerate}
  \item
$z\in C^{2}(\Omega)$, where $\Omega\subset \R^2$ is some punctured disk (of any radius) centered at $(x_0,y_0)$.
 \item
$z\in C^0(\overline{\Omega})$, with $z(x_0,y_0)= z_0$.
 \item
The Hessian determinant $z_{xx} z_{yy} -z_{xy}^2$ does not vanish around $(x_0,y_0)$.
 \end{enumerate}
Here, we identify two elements of $\cA_1$ if they coincide on a neighborhood of $(x_0,y_0)$. 

Then, there exists an explicitly constructed bijective correspondence $\cA_1 \mapsto \cA_2$ between the class $\cA_1$ and the class $\cA_2$ of $2\pi$-periodic, real analytic, nowhere vanishing functions $A(u):\R\flecha \R\setminus \{0\}$.
\end{teorema}

Let us make some comments about this theorem. The correspondence $\cA_1 \mapsto \cA_2$ in the statement of Theorem \ref{main1} associates to each solution $z(x,y)$ a function $A(u)\in \cA_2$ that describes the asymptotic behavior of the solution at the singularity, when we parametrize the graph $z=z(x,y)$ with respect to adequate conformal coordinates $(u,v)$. The injectivity of the bijective correspondence $\cA_1 \mapsto \cA_2$ implies that $A(u)$ uniquely determines the solution $z(x,y)$, while surjectivity gives a general existence theorem for non-removable isolated singularities to \eqref{maineq}.

We should note that the real analyticity of $\cH$ is used in the proof in order to ensure the previous existence and uniqueness properties. However, in the proof of Theorem \ref{main1} we will also provide a detailed description of the asymptotic behavior of the solution $z$ around a non-removable singularity, in the general case that $\cH$ is only of class $C^{k,\alfa}$.
 
The fact that $z\in C^0(\overline{\Omega})$ in condition (2) of the statement of Theorem \ref{main1} follows from ellipticity. Hence, this second condition simply prescribes, for definiteness, that the continuous extension to $z$ to the puncture has value $z_0$.

Condition (3) in the statement of Theorem \ref{main1} is equivalent to the property that the Gaussian curvature of the graph $z=z(x,y)$ does not vanish around the singularity. We should point out that all known examples of non-removable isolated singularities of \eqref{maineq} (in particular, all examples given by the existence part in Theorem \ref{main1}) have the property that their Gaussian curvature blows up at the singularity. In this sense, it is tempting to conjecture that Theorem \ref{main1} should still hold without this curvature condition.

The proof of Theorem \ref{main1} is inspired by ideas developed by the authors in their previous works \cite{GJM1,GJM2} on isolated singularities of elliptic Monge-Ampère equations, and depends on a blend of ideas from surface theory, complex analysis and elliptic theory. The basic idea is to reparametrize the graph of any solution to \eqref{maineq} with respect to adequate conformal parameters, so that the punctured disk where the graph is defined transforms under this reparametrization to an annulus (with one boundary component of the annulus corresponding to the singularity), and so that \eqref{maineq} transforms under this process to a quasilinear elliptic system. In these new coordinates, we will solve a Cauchy problem for this quasilinear elliptic system with adequate initial data, and the solution will provide the desired results about existence, uniqueness and asymptotic behavior of the solution.

We should also observe that there are crucial differences of the situation treated in this paper with respect to our previous works \cite{GJM1,GJM2} on elliptic Monge-Ampère equations. In fact, even though \eqref{maineq} is quasilinear, it is not uniformly elliptic. More specifically, in our situation the uniform ellipticity condition is lost precisely at the isolated singularity, and this is a source of considerable complication, specially in the characterization of the conformal structure of a non-removable isolated singularity to \eqref{maineq}.

Another difference with respect to previous works concerns the role of limit cones. Indeed, any elliptic solution to the Monge-Ampère equation ${\rm det}(D^2u)= f(x,u,Du)$ has positive curvature, and it converges to a limit convex cone at any non-removable isolated singularity, see \cite{GJM1}. This cone describes the asymptotic behavior of the solution at the singularity. In contrast, any solution to \eqref{maineq} is asymptotic to the light cone, so in this case the \emph{tangent cone} at the singularity does not distinguish different solutions of \eqref{maineq}. In this sense, some other object other than tangent cones was needed in order to describe the moduli space of non-removable isolated singularities of \eqref{maineq}.

We have organized the paper as follows. In Section \ref{sec:pre} we introduce basic geometric notions and equations regarding spacelike surfaces of prescribed mean curvature in $\L^3$, when parameterized with respect to conformal parameters. 

In Section \ref{sec:rem} we prove that if $z(x,y)$ is a solution to \eqref{maineq} with an isolated singularity at the origin, and whose Gaussian curvature does not vanish around the singularity, then the singularity is removable if and only if the associated conformal structure of the graph $z=z(x,y)$ is that of a punctured disk. This is the only place of the proof where the curvature condition (3) in Theorem \ref{main1} is needed. 

In Section \ref{sec:nonr} we will introduce a special conformal parametrization $(u,v)$, defined on a horizontal quotient strip, around any non-removable isolated singularity of \eqref{maineq}, and we will show that the asymptotic behavior of the solution around the singularity depends on a certain periodic real function $A(u)$. This will allow us to define the map $\cA_1\mapsto \cA_2$ in the statement of Theorem \ref{main1}. 

In Section \ref{sec:exis} we will show how to reverse this process, by constructing from a real analytic periodic real function $A(u)$ an associated solution to \eqref{maineq} with a non-removable isolated singularity. Finally, in Section \ref{sec:class} we will gather the results of the previous sections to prove Theorem \ref{main1}, and we will discuss some consequences and particular examples.

\section{Surfaces of prescribed mean curvature in $\L^3$}\label{sec:pre}

Let $\L^3$ denote the Minkowski three-dimensional space, i.e. the vector space $\R^3$ endowed with the Lorentzian metric
\beq\label{L3}\langle \cdot,\cdot\rangle=dx^2+dy^2-dz^2\eeq in canonical coordinates $(x,y,z)$ of $\R^3$. An immersion $\psi:M^2\flecha \L^3$ from a   surface into $\L^3$ is called \emph{spacelike} if the metric $ds^2:=\psi^*(\langle \cdot,\cdot\rangle)$ induced on $M$ via $\psi$ is Riemannian. Equivalently, the surface is spacelike if its tangent plane is a spacelike plane of $\L^3$ at every point. 

A spacelike surface in $\L^3$ is locally a graph with respect to any timelike direction of $\L^3$; in particular, we may view $\psi$ around each point $p\in M$ as a map $\psi(x,y)=(x,y,z(x,y))$ where the graph function $z$  satisfies $z_x^2+z_y^2<1$ at every point due to the spacelike condition.

If $\psi:M\flecha \L^3$ is a spacelike surface, we will denote its upwards-pointing unit normal vector field as $G=(G_1,G_2,G_3):M\rightarrow\H^2=\{(x,y,z)\in\L^3: \,z>0,\, x^2+y^2-z^2=-1\}$. Note that $G$ takes values in the hyperbolic space $\H^2$, and so it is a unit timelike vector field at every point. The \emph{Gauss map} of $\psi$ is given by the composition 
 \begin{equation}\label{defga}
g=\pi\circ G:M\rightarrow \D:=\{(x,y): x^2+y^2<1\}, 
 \end{equation}
 where $\pi:\H^2\rightarrow \D$ is the stereographic projection in $\L^3$, given by, $\pi(x,y,z)=(x-iy)/(1+z)$.

 If $\psi:M\flecha \L^3$ is a spacelike surface, we can regard naturally $M$ as a Riemann surface with the conformal structure determined by its induced metric $ds^2$. Let $w=u+iv$ be a conformal parameter on $M$, so that $ds^2=\Lambda |dw|^2$ for some positive function $\Lambda$. It is well known then that $\psi$ satisfies in terms of the conformal parameters $(u,v)$ the quasilinear elliptic system
 \beq\label{lapinm}\psi_{uu}+\psi_{vv}=2H\psi_{u}\times\psi_{v},\eeq
 where $\times$ is the standard cross product in $\L^3$, given by $\esiz a\times b,c\esde = -{\rm det} (a,b,c)$, and $H$ is the mean curvature function of the surface.

In \cite{AkNi} Akutagawa and Nishikawa studied the Gauss map of a conformally parametrized spacelike surface in $\L^3$ in terms of its Gauss map. In particular, if $w=u+iv$ denotes a complex conformal parameter for a spacelike surface $\psi:M\flecha \L^3$, it was proved in \cite{AkNi} that the Gauss map $g$ and the mean curvature $H$ of $\psi$ satisfy the complex elliptic PDE
 \beq\label{genharmonic}
 H\left(\frac{\partial^2 g}{\partial w\overline{w}}+\frac{2\overline{g}}{1-|g|^2}\frac{\partial g}{\partial w}\frac{\partial g}{\partial \overline{w}}\right)=\frac{\partial H}{\partial w}\frac{\partial g}{\partial\overline{w}}.
 \eeq
 
In particular, if $H$ is a constant different from zero,   \eqref{genharmonic} reduces to
 \beq\label{harmonic}   \frac{\partial^2 g}{\partial w\overline{w}}+\frac{2\overline{g}}{1-|g|^2}\frac{\partial g}{\partial w}\frac{\partial g}{\partial \overline{w}} =0,\eeq that is, $g:M\flecha \D$ becomes a harmonic map into the Poincaré disk.
 
Akutagawa and Nishikawa also gave in \cite{AkNi} a Weierstrass-type representation formula for conformally immersed spacelike surfaces $\psi:M\flecha \L^3$ of non-zero mean curvature that allows to recover the immersion $\psi$ in terms of its Gauss map $g$ and its mean curvature function $H$ as
\beq\label{rep}
  \def\arraystretch{1.5}\begin{array}{l}x_w=\displaystyle{\frac{\overline{g}_w(1+g^2)}{H(1-|g|^2)^2}}\\
  y_w=\displaystyle{-i\frac{\overline{g}_w(1-g^2)}{H(1-|g|^2)^2}}\\
  z_w=\displaystyle{\frac{2\overline{g}_wg}{H(1-|g|^2)^2}}.
  \end{array} 
 \eeq

Assume now that the spacelike surface $\psi:M\flecha \L^3$ is a graph $z=z(x,y)\in C^3(U)$ over some planar domain $U\subset \R^2$ (this is the case if $z(x,y)$ is a $C^2$ solution to (\ref{maineq}) and $\cH=\cH(x,y,z)$ is $C^{1,\alpha}$ by \cite{Ni}). In this situation, by the uniformization theorem for nonanalytic metrics \cite{Sau} we know that the change of coordinates 
\begin{equation}\label{tristar}
\Phi:U\rightarrow  \mathcal{G}:=\Phi(U)\subset \R^2 ,\qquad
(x,y)\mapsto \Phi(x,y)=(u(x,y),v(x,y)),
\end{equation}
 that relates the parameters $(x,y)$ with the conformal coordinates $(u,v)$ is a $C^{2}$ diffeomorphism with positive Jacobian. Also, if we denote
 \beq\label{coef}
a=1-z_y^2, \hspace{0.4cm} b= z_xz_y , \hspace{0.4cm}
c=1-z_x^2, \hspace{0.4cm} \Sigma= ac-b^2 = 1-z_x^2-z_y^2>0,\eeq
and view $a,b,c,\Sigma$ as functions depending on the coordinates $(u,v)$ via the inverse change $\Phi^{-1}$, the following equivalent Beltrami systems are satisfied
\vspace{0.2cm}
 
 \beq\label{sistxy} \left(
 \begin{array}{c}x_v\\y_v
  \end{array}\right)=\frac{1}{\sqrt{\Sigma}}\left(
 \begin{array}{cc}b &-a\\c &-b 
 \end{array}\right) \left(
 \begin{array}{l}x_u\\y_u 
 \end{array}\right), \qquad \left(
 \begin{array}{c}y_u\\y_v
  \end{array}\right)=\frac{1}{a}\left(
 \begin{array}{cc}b &-\sqrt{\Sigma}\\ \sqrt{\Sigma} &b 
 \end{array}\right) \left(
 \begin{array}{l}x_u\\x_v 
 \end{array}\right).
\eeq 

   Note that if $U\subset \R^2$ is a punctured disk,   $\mathcal{G}$ will be a domain in $\R^2\equiv \C$ that is conformally equivalent to either the punctured disk or a certain annulus.

\section{Removability of isolated singularities of prescribed mean curvature}\label{sec:rem}

 \begin{definicion}
Let $\cH$ be a function defined on an open set $\cO\subset\L^3$, and $\psi:M\flecha \cO\subset\L^3$ be a $C^2$ spacelike surface in $\L^3$ with upwards-pointing unit normal. We say that $\psi$ has \emph{prescribed mean curvature $\cH$} if the mean curvature function $H$ of $\psi$ satisfies $H(p)=\cH(\psi(p))$ for every $p\in M$.
 \end{definicion}

Note that if a spacelike graph $(x,y,z(x,y))$, $(x,y)\in U\subset \R^2$, has as a prescribed mean curvature a locally $C^{k,\alpha}$ function $\cH$ then $z(x,y)$ is a solution to \eqref{maineq} that satisfies the ellipticity condition $z_x^2+z_y^2 <1$ on $U$.
 
From now on we will work with solutions $z(x,y)$ to  \eqref{maineq} in a punctured neighborhood of the origin. Observe that in this case the ellipticity condition $z_x^2+z_y^2<1$ implies that $z(x,y)$ has a continuous extension to the origin. Hence we will suppose that $z(0,0)=0$, the open set $\cO\subset \L^3$   contains the origin and $z(x,y)$ is defined over a 
 punctured disk $\Omega=\{(x,y): 0<x^2+y^2<\rho^2\}$.
 
  The following lemma concerns the asymptotic behavior of the gradient of $z$.

\begin{lema} \label{univ}
Let $z\in C^2(\Omega)$ be an elliptic solution to \eqref{maineq} defined on the punctured disk $\Omega\subset \R^2$, and assume that the Gaussian curvature of the spacelike graph $z=z(x,y)$ in $\L^3$ does not vanish around the singularity.  If the singularity at the puncture is non removable, then the   map $ \mu$ into the open unit disk given by $ \mu (x,y)=(z_x(x,y),z_y(x,y))$ is univalent and proper in a  closed  punctured neighborhood of the origin.
\end{lema}
\begin{proof}

Since the Gaussian curvature of the graph $z=z(x,y)$ does not vanish around the origin, it follows that $ \mu$ is a local diffeomorphism in a punctured neighborhood of the origin. On the other hand, as observed in the previous section, $z\in C^3(\Omega)$, and
since the singularity of $z$ is not removable, a theorem by Bartnik (cf. \cite{Bar})   states that  the graph of $z$  must be asymptotically tangent to the  upper null cone $\n^2_+:=\{(x,y,z)\in\L^3:\ z>0,\ x^2+y^2-z^2=0\}$ or to the lower null cone $\n^2_-:=\{(x,y,z)\in\L^3:\ z<0,\ x^2+y^2-z^2=0\}$ at the singularity. That is, $\mu$ is proper into the open unit disk.

Therefore, there exists an $R>0$ and an annulus $ A =\{(p,q)\in\r^2: R^2<p^2+q^2<1\}$ such that the map $\mu:\mu^{-1}(A)\fl A$ is a covering map.

  
Finally we are going to prove that the number of sheets of the covering $\mu$ is equal to $1$.  As mentioned before, the graph of $z$ is tangent to one of the null cones at the singularity, so assume, for instance, it is tangent to $\n^2_+$. Then, for all $\varepsilon >0$ small enough, the curve $\gamma_{\varepsilon}=\{(x,y,z(x,y)):\ z(x,y)=\varepsilon \}$ is embedded and its degree is equal to $1$. Moreover, the vertical projection of the upward unit normal $G(\gamma_{\varepsilon})=(\frac{\mu}{\sqrt{1-|\mu|^2}},\frac{1}{\sqrt{1-|\mu|^2}})(\gamma_{\varepsilon})$ to the graph is perpendicular to the planar curve $\gamma_{\varepsilon}$. Thus, the degree of $\mu/|\mu|$ is equal to 1 and an elementary topological argument gives us that the number of sheets must be 1.
\end{proof}
 
Let $ds^2$ denote  the induced metric of the graph $z=z(x,y)$ as a spacelike surface in $\L^3$. By uniformization, $(\Omega,ds^2)$ is conformally equivalent to either the punctured disk $\D^*$, or to some annulus $\A_r =\{\zeta\in \C: 1<|\zeta|<r\}.$

The next lemma is a removable singularity result. It relates the conformal structure of $(\Omega,ds^2)$ to the possibility that the solution $z(x,y)$ extends smoothly across the puncture.

 \begin{lema}\label{anillo}
 Let $z\in C^2(\Omega)$ be an elliptic solution to \eqref{maineq} defined on the punctured disk $\Omega\subset \R^2$, and assume that the Gaussian curvature of the spacelike graph $z=z(x,y)$ in $\L^3$ does not vanish around the singularity. If $(\Omega ,ds^2)$ is conformally equivalent to a punctured disk, then the function $z(x,y)$ extends $C^2$-smoothly across the puncture. That is, $z$ has a removable singularity at the puncture. 
 \end{lema}
 
 \begin{proof}
First, note that we can choose $\rho>0$ in the definition of $\Omega$ such that $z$ is also $C^2$ at its exterior boundary.  

Let us denote $p=z_x$, $q=z_y$, $r= z_{xx}$, $s=z_{xy}$, $t=z_{yy}$. 
 
We argue by contradiction. Assume that $z$ has a non-removable isolated singularity at the origin, and that  $(\Omega ,ds^2)$ is conformal to the punctured disk $\D^*$. Define the map $\hat{\mu}:(u,v)\mapsto (p,q)(u,v):\D^*\flecha \R^2$. We prove next that $\hat{\mu}\in  W^{1,2}(\D^*)$. 

By using   \eqref{maineq} and \eqref{sistxy} we have that
$$|\nabla \hat{\mu}|^2=\frac{x_uy_v-x_vy_u}{\sqrt{1-p^2-q^2}}\left(\cH(r+t) (1-p^2-q^2)^{3/2}+(2-p^2-q^2)(-rt+s^2)\right).$$
Let $\Omega ' \Subset\Omega$ be a domain and consider  $\varepsilon>0$ such that     $\Omega '\subset\mathcal{A}_{\varepsilon}:=\{(x,y):\ \varepsilon<x^2+y^2<\rho^2\}$, then
 we can estimate  
\vspace{0.2cm} 
 \beq\label{int}\begin{array}{l}
   \displaystyle{\int_{\Omega'}\frac{1}{\sqrt{1-p^2-q^2}}\left(\cH(r+t) (1-p^2-q^2)^{3/2}+(2-p^2-q^2)(-rt+s^2)\right)dxdy}\\
   \\
    \leq \displaystyle{\int_{\mathcal{A}_{\varepsilon}} \cH (r+t) (1-p^2-q^2) dxdy+\int_{\mathcal{A}_{\varepsilon}} \frac{2-p^2-q^2}{\sqrt{1-p^2-q^2}}(-rt +s^2) dxdy}.
 
 \end{array}\eeq
  By Stokes' theorem applied to the vector field $\cH(1-p^2-q^2)(p,q)$  we obtain
 $$ 
 \displaystyle{  \int_{\mathcal{A}_{\varepsilon}} \cH (r+t) (1-p^2-q^2)dxdy    } =$$
 $$=\displaystyle{\int_{\partial \mathcal{A}_{\varepsilon}} \cH(1-p^2-q^2)\langle (p,q),\nu\rangle
dl -\int_{\mathcal{A}_{\varepsilon}}\langle (p,q),\nabla (\cH (1-p^2-q^2))\rangle dxdy}$$
$$\leq\displaystyle{C_0-\int_{\mathcal{A}_{\varepsilon}}(p\cH_x+q\cH_y)(1-p^2-q^2)dxdy +2\int_{\mathcal{A}_{\varepsilon}}\cH(p^2r+2pqs+q^2t)dxdy}$$
 for a certain constant $C_0>0$, where $\nu$ is the   exterior unit normal to $\partial \mathcal{A}_{\varepsilon}$ and we have estimated $|\langle (p,q),\nu\rangle|\leq 1$.
 
On the other hand,  by \eqref{maineq},
$$\displaystyle{2\int_{\mathcal{A}_{\varepsilon}}\cH(p^2r+2pqs+q^2t)dxdy}=$$
$$=4\int_{\mathcal{A}_{\varepsilon}}\cH^2(1-p^2-q^2)^{3/2}dxdy-2\int_{\mathcal{A}_{\varepsilon}}\cH(r+t)(1-p^2-q^2).$$
With all this we deduce that 

$$ 
 \displaystyle{  \int_{\mathcal{A}_{\varepsilon}} \cH (r+t) (1-p^2-q^2)dxdy    } \leq$$
 $$\leq\displaystyle{\frac{C_0}{3}  -\frac{1}{3}\int_{\mathcal{A}_{\varepsilon}}(p\cH_x+q\cH_y)(1-p^2-q^2)dxdy  +\frac{4}{3}\int_{\mathcal{A}_{\varepsilon}}\cH^2(1-p^2-q^2)^{3/2}dxdy.}$$
 
   Thus, since $\cH$ is smooth in $\cO$, if we let $\varepsilon\to 0$ and recall that $p^2+q^2<1$ by ellipticity, we may conclude that the first integral in the last line of \eqref{int} is bounded from above by a constant independent of $\Omega'$.
  
  On the other hand,    by Lemma \ref{univ}, we can use a change of variables to get
$$
\left|\int_{\mu^{-1}(A)} \frac{2-p^2-q^2}{\sqrt{1-p^2-q^2}}(-rt +s^2) dxdy\right|=  \int_{ A } \frac{2-p^2-q^2}{\sqrt{1-p^2-q^2}}dpdq= \frac{(4-R^2)}{3}<+\infty,
$$
where $A$ is the annulus $\{(p,q)\in\r^2: R^2<p^2+q^2<1\}$ given by the proof of Lemma \ref{univ}. In particular, the second integral in the last line of \eqref{int} is bounded by an universal constant. 
  
This proves that $\displaystyle{\int_{\D^*}|\nabla\hat{\mu}|^2dudv}$ is bounded. Therefore, we can apply the Courant-Lebesgue oscillation theorem \cite[Lemma 3.1]{Cou} to deduce that there exists a sequence $\{\tau_n\}_{n=1}^{\infty}\subset\R$ with $\tau_n\downarrow 0$ for $n\rightarrow\infty$ such that 
  $$
   \underset{n\rightarrow \infty}\lim\int_{C(0,\tau_n)}| d\hat{\mu}(u,v)|=0,
   $$
where $C(0,\tau_n)$ denotes the circle of radius $\tau_n$ centered at the origin. 
 That is, the sequence of closed curves $\{\hat{\mu}(C(0,\tau_n))\}$ contained in the unit disk has a subsequence collapsing into a point $(p_0,q_0)$ of the closed unit disk when $\tau_n$ goes to zero. This contradicts Lemma \ref{univ} and completes the proof of Lemma \ref{anillo}.
 \end{proof}
 

\section{Non-removable isolated singularities}\label{sec:nonr}

In this section we will study the asymptotic behavior at an isolated singularity of elliptic solutions to \eqref{maineq}. More specifically, along this section, $z\in C^2(\Omega)$ will be an elliptic solution to \eqref{maineq} with $z(0,0)=0$. We will suppose that the conformal type of the Riemannian metric $(\Omega, ds^2)$ of the spacelike graph $z=z(x,y)$ in $\L^3$ is that of an annulus (i.e. $(\Omega, ds^2)$ is not conformally equivalent to a punctured disk). 
 
  Observe that, since  the conformal structure is  that of an annulus then $z(x,y)$ does not extend $C^2$-smoothly to the origin. Also, note that, by Lemma \ref{anillo}, the condition of being a conformal annulus is guaranteed when   $z$ satisfies that $z_{xx} z_{yy}-z_{xy}^2 $ does not vanish around the puncture, and   $z$ does not extend smoothly to the the origin.

Let $(u,v)$ denote conformal parameters for $(\Omega,ds^2)$; since $(\Omega,ds^2)$ is conformally an annulus, we may assume that $(u,v)$ vary in the conformal quotient strip $$\Gamma_R=\{w\in \C: 0<\Im w < R\} /(2\pi \Z)$$ for a certain $R>0$. In this way, we can parametrize the graph $\{(x,y,z(x,y)): (x,y)\in \Omega\}$ as 
 \begin{equation}\label{conpar}
\psi(u,v)=(x(u,v),y(u,v),z(u,v)):\Gamma_R\flecha \L^3.
 \end{equation}
Note that $\psi$ extends continuously to $\R/(2\pi \Z)$, with $\psi(u,0)=(0,0,0)$ for every $u$.   Thus,  we are identifying the isolated singularity at the origin with the   circle $\{(u,v) : v =0\}/(2\pi \Z) $ in the conformal parametrization. As explained in Section \ref{sec:pre}, this identification is defined by means of the diffeomorphism $\Phi$ in \eqref{tristar}. For convenience, the union of the annulus $\Gamma_R$ with its   boundary given by $\R/(2\pi\Z)$ will be denoted by $\Gamma_R\cup\R$.

\subsection{A boundary regularity lemma}\label{sub:bound}

Observe that since $\psi$ is a conformal immersion of prescribed mean curvature $\cH$, we get from \eqref{lapinm} that $\psi$ satisfies the elliptic quasilinear system
 \begin{equation}\label{lapi2}
 \Delta \psi = 2(\cH \circ \psi) \, \psi_u \times \psi_v.
 \end{equation}
We use \eqref{lapi2} to prove the following boundary regularity lemma:


\begin{lema}\label{extanalitica}
In the conditions above, assume that $\cH\in C^{k,\alpha}(\cO)$, $k\geq 1$ (resp. $\cH\in C^{\omega}(\cO)$). Then, the conformal immersion $\psi(u,v):\Gamma_R\flecha \L^3$ extends to $\Gamma_R\cup \R$ as a map of class $C^{k+2}$  (resp. as a $C^{\omega}$ map).
\end{lema}

\begin{proof} 
It is well-known that since $\psi(u,v)$ is of class $C^2(\Gamma_R)$ and $\cH$ is of class $C^{k,\alpha}(\cO)$,   from the elliptic equation \eqref{lapinm} we have that $\psi(u,v)$ is in fact of class $C^{k+2}(\Gamma_R)$.

In order to check  the differentiability at the boundary we will follow a bootstrapping method.
Consider an arbitrary point of $\R$, which we will suppose without
loss of generality to be the origin. Also, consider for $0<\delta<R$
the domain $\D^+=\{(u,v): 0<u^2+v^2<\delta^2\}\cap \Gamma_R$.
 
The inequalities $$ (x_u -y_v)^2+(x_v+y_u)^2\geq 0,\quad (x_u -z_v)^2+(x_v+z_u)^2\geq 0,\quad (y_u -z_v)^2+(y_v+z_u)^2\geq 0, $$ lead respectively 
to $$\def\arraystretch{2} \begin{array}{l} x_uy_v-x_v y_u\leq \frac{1}{2}(|\nabla x|^2+|\nabla y|^2),\\
x_uz_v-x_v z_u\leq \frac{1}{2}(|\nabla x|^2+|\nabla z|^2),\\
y_uz_v-y_v z_u\leq \frac{1}{2}(|\nabla y|^2+|\nabla z|^2).\end{array}$$
From here, \eqref{lapi2}
  and the fact that $\cH$ is bounded
  yield \beq\label{ineqlaplac}|\Delta
\psi|\leq c |\nabla \psi|^2  \eeq for a certain constant
$c>0$.  Thus, we can apply   Heinz's Theorem in
\cite{He2} to deduce that $\psi\in
C^{1,\alpha}(\overline{\D^+_{\varepsilon}})$  for all
$\alpha\in(0,1)$, where $\D^+_{\varepsilon}=\{(u,v)\in\R^2:\ u^2+v^2<\varepsilon^2,\ v>0\}$
 for a certain $0<\varepsilon<\delta$.

  Hence the right-hand side of \eqref{lapi2} is of class $\C^{0,\alpha}(\overline{\D^+_{\varepsilon}})$, from where a
standard potential analysis argument (cf. \cite[Lemma 4.10]{GiTr})
ensures that $\psi\in C^{2,\alpha}(\overline{\D^+_{\varepsilon/2}})$. A recursive process proves then that $\psi$ is of class $C^{k+2,\alpha}$ at a neighborhood of the origin. As we can
do the same argument for all points of $\R$ and not just at the origin,
we conclude that $\psi\in C^{k+2} (\Gamma_R\cup \R)$. In particular, if $\cH\in C^{\8}$ we have that $\psi\in C^{\8} (\Gamma_R\cup \R)$.

To finish the proof, let us note that $\psi\in C^{\omega} (\Gamma_R\cup \R)$ if $\cH\in C^{\omega}(\cO)$.  This is a direct consequence of \eqref{lapi2} and \cite[Theorem 3]{Mu2} (specifically, our situation is covered by the case $m=0$ in Theorem 3 of \cite{Mu2}).
\end{proof}



\subsection{Study of the limit null curve}\label{sub:null}

Let us denote by $b(u)$ the $2\pi$-periodic curve $\psi_v(u,0)$. It follows from Lemma \ref{extanalitica} that $b(u)$ is $C^{k+1}-$smooth (resp.   analytic) if   $\cH\in C^{k,\alpha}(\cO)$ (resp. if  $\cH$ is   analytic). Besides, since $(u,v)$ are conformal parameters for the induced metric of the surface and $\psi(u,0)=(0,0,0)$, we have $\esiz \psi_v,\psi_v \esde =\esiz \psi_u,\psi_u\esde =0$ at points of the form $(u,0)$. That is, $b(u)$ takes its values in $\N_+^2\cup \{{\bf 0}\}\cup \N_-^2$.

 \begin{definicion}\label{nulcu} We will call $b(u):\R/(2\pi \Z)\flecha \N_+^2\cup \{{\bf 0}\}\cup \N_-^2$  the \emph{limit null curve} at the singularity of the graph $z=z(x,y)$ associated to the conformal parameters $(u,v)$.\end{definicion}
 
The next result analyzes the geometric properties of this limit null curve.
\begin{lema}\label{bu}
The \emph{limit null curve} $b(\R)$ is a closed, spacelike Jordan curve in $\N_+^2$ or in $\N_-^2$.
\end{lema}
\begin{proof}
  Using \eqref{sistxy} we obtain
 $$z_u^2+z_v^2=p^2(x_u^2+x_v^2)+q^2(y_u^2+y_v^2)+2pq(x_uy_u+x_vy_v)=\frac{p^2+q^2}{\sqrt{1-p^2-q^2}}(x_uy_v-x_vy_u),$$ where we are denoting $p=z_x$, $q=z_y$. Hence, from \eqref{lapi2} we deduce that
 \beq\label{lapz}
 \Delta z =\frac{2H\sqrt{1-p^2-q^2}}{p^2+q^2}(z_u^2+z_v^2).\eeq
 Observe that from Lemma \ref{univ} there exists $\varepsilon>0$ such that $p^2+q^2\geq\varepsilon>0$ close to the origin. Moreover, because of Lemma \ref{extanalitica} we obtain from \eqref{lapz} that
 \beq\label{lapz2} |\Delta z|\leq C|\nabla z|,\eeq for a certain constant $C>0$.
 We can now prove that $b(u) \neq \bf{0}$ $\forall u\in\R$. Arguing by contradiction, assume there exists a point $u_0\in\R$ such that $b(u_0)=\bf{0}$. Then, $z_v(u_0,0)=0$. Also, note that $z(u_0,0)=0=z_u(u_0,0)$. 
 
 From here, a standard nodal result (cf. Corollary $2$ in \cite{Sch}) applied to \eqref{lapz2} leads to the existence of two crossing nodal curves for $z(u,v)$ at $(u_0,0)$. One of these curves could be the real axis, which corresponds to the isolated singularity in this conformal parametrization. The existence of a second nodal curve gives a contradiction with the fact that the graph $z=z(x,y)$ converges asymptotically to either  $\N_+^2\cup \{{\bf 0}\}$ or to $\N_-^2\cup \{{\bf 0}\}$ at the origin, by Bartnik's theorem (cf. \cite{Bar}). This proves that $b(u)\neq \bf{0}$ for every $u\in \R$.
   
Next we will prove that $b(u)$ is a regular spacelike curve. Denoting $w=u+iv$,   let us define 
\beq \label{omegag2}
\vartheta(w,\overline{w}):=\frac{g_{\overline{w}}}{(1-|g|^2)^2},\qquad F(w,\overline{w}):=|g|^2-1.
\eeq
From \eqref{rep} we have
   \beq\label{omegag}\overline{\vartheta}=\frac{H}{2}(x_w+iy_w),\qquad g=\frac{x_w-iy_w }{z_w}.\eeq 
  Hence,  as $b(u)\in\N^+\cup\N^-$ and so $2z_w (u,0)=-iz_v (u,0)\neq 0$ for all $u\in \R$, we have that $\vartheta$, $g$ and $F$ are well defined in a neighborhood of the real axis.

 On the other hand, a simple computation using \eqref{genharmonic} leads to 
 $$
 F_{w\overline{w}}=\frac{|F_{\overline{w}}|^2}{|g|^2}+F^2\left(\vartheta\overline{g}\left( \frac{H_w}{H}-\frac{g_w}{g}\right)+\overline{\vartheta}g\left(\frac{H_{\overline{w}}}{H}-\frac{\overline{g_w}}{\overline{g}}\right)\right)-2F(\overline{g}^2g_w\vartheta +g^2\overline{g_w}\overline{\vartheta}),
 $$
 which by  Lemma \ref{extanalitica},   implies that
 \beq\label{lapF2}  |F_{w\overline{w}}|\leq C( |F_{\overline{w}}|+|F|),
 \eeq
 for a certain constant $C>0$.
 Moreover, from \eqref{lapi2} and \eqref{omegag} we have the relation 
\beq\label{gz} 4|F_w z_w|^2  (u,0) =4|g_w z_w|^2  (u,0) =\langle b'(u), b'(u)\rangle,\quad \forall u\in\R.\eeq

 Thus, if there existed some $u_0\in\R$ such that $\langle b'(u_0), b'(u_0)\rangle= 0$,  since $z_w (u_0,0)\neq 0$ we  deduce from   \eqref{gz} that  $F_w(u_0)=0$. 
 
 In that case, as   $F(u,0)=0$ for every $u\in \R$,  nodal theory (cf. Corollary $2$ in \cite{Sch}) applied to   \eqref{lapF2} would guarantee the existence of two crossing nodal curves of $F$ at $u_0$. This would contradict the fact that the singularity at the origin is isolated, and more specifically, the fact that $\psi(u,v)$ is a spacelike immersion in $\Gamma_R$ .

Finally, we are going to prove that the limit null curve $b(u)$ is embedded.
 
 We deduce from Lemma \ref{univ} that $g(u,0):\R/(2\pi \Z)\flecha \S^1\subset \C$ is injective, where $\S^1$ denotes the complex numbers of modulus one. Denoting $b=(b_1,b_2,b_3)$, the second formula in \eqref{omegag} shows that 
 \begin{equation}\label{sefor}
g(u,0)=\frac{b_1(u)-ib_2(u)}{b_3(u)},\end{equation} from where we deduce that $b(u)=b_3(u)(\overline{g},1)$ is a closed embedded curve since $b_3(u)\neq 0$.
 This concludes the proof of Lemma \ref{bu}.
 \end{proof}
 \begin{nota}\label{rem:K}
In the situation above, since $\vartheta$ is well defined along the real line, we have from    (\ref{omegag2})  that $g_{\overline{w}}$ vanishes identically at $\R$. In addition, since $|g_w(u,0)|\neq0$ from (\ref{gz}), we obtain that the Gaussian curvature of the graph (cf. \cite{AkNi})
 \beq\label{kkk}
 K=H^2\left(\frac{|g_w|^2}{|g_{\overline{w}}|^2}-1\right)
 \eeq
goes to $+\infty$ at the puncture.
 \end{nota}
 \subsection{The canonical conformal parametrization}
 Let us point out that the conformal parametrization \eqref{conpar} for the graph $z=z(x,y)$ is not unique. Indeed, assume that $\zeta=\zeta(w)$ is a $2\pi$-periodic biholomorphism $\zeta(w): \cV\subset \Gamma_R\flecha \Gamma_{R'}$ between a region $\cV\subset \Gamma_R$ such that exists $\Gamma_{R''}\subset \cV$, for some $R''>0$ small enough, and some other $\Gamma_{R'}$, $R'>0$. Then, $\psi^*:= \psi\circ \zeta^{-1}:\Gamma_{R'}\flecha \L^3$ is also a conformal parametrization of the graph $z=z(x,y)$ around its isolated singularity at the origin. 
 
 In addition, it can be easily checked that the limit null curve $b(u)$ at the singularity (see Definition \ref{nulcu}) associated to some conformal parameters $(u,v)$ really depends on this choice of a conformal parametrization for the graph $z=z(x,y)$.
 
Nonetheless,  in the case that $z(x,y)$ is real analytic, we can use Lemma \ref{bu} to fix a specific conformal parametrization for its graph. Specifically, in the proof of Lemma \ref{bu} we showed that for any conformal parametrization \eqref{conpar}, the Gauss map at the singularity in these conformal parameters, given by $g(u,0)$, is a real analytic bijective map $$g(u,0):\R/(2\pi \Z) \flecha \S^1\subset \C$$ with $|g'(u,0)|\neq 0$ for every $u$. Let $s=s(u)$ be the reparametrization of $g(u,0)$ such that $g(u,0)=\tilde{g}(s(u),0)$ where $\tilde{g}(s,0)=\cos s + i \sin s$, and let $\zeta(w)$ denote the holomorphic extension of $s(u)$, i.e. the unique holomorphic function with $\zeta(u,0)=s(u)$ for every $u\in \R$. We note that $\zeta(u)$ exists by real analyticity of $g(u,0)$. Moreover, 
 by the $2\pi$-periodicity of $g(u,0)$ it is clear that $\zeta(w)$ defines a biholomorphism between some $\cV\subset \Gamma_R$,  with $\Gamma_{R''}\subset \cV$ for $R''>0$ small enough, and some other $\Gamma_{R'}$, $R'>0$. In other words, we can choose the conformal parameters $(u,v)$ in \eqref{conpar} in such a way that $g(u,0)=\cos u + i \sin u$. Besides, it easily follows from \eqref{sefor} that in these conditions the limit null curve $b(u)$ of the graph $z=z(x,y)$ with respect to this specific conformal parametrization is given by $$b(u)= A(u)(\cos u, -\sin u,1):\R/(2\pi\Z)\flecha \N_+^2 \cup \N_-^2,$$ for some real analytic, nowhere vanishing $2\pi$-periodic function $A(u)$ (the fact that $A(u)\neq 0$ comes directly from the fact proven above that $b(u)\neq {\bf 0}$ at every $u\in \R$).
 
 As a consequence we obtain the next lemma.
 
 \begin{lema}\label{cano}
 Assume $\cH\in C^{\omega}(\cO)$, $\cH>0$, and let $z\in C^{2}(\Omega)$ be an elliptic solution to \eqref{maineq} whose associated conformal structure is that of an annulus. Then, $z\in C^{\omega}(\Omega)$ and there exists a unique $2\pi$-periodic conformal parametrization $$\psi(u,v):\Gamma_R\flecha \L^3$$ of the graph $z=z(x,y)$ around the origin such that $\psi$ extends analytically to $\Gamma_R\cup \R$ with $$\psi_v(u,0)= A(u) \, (\cos u, -\sin u, 1)$$ for some nowhere vanishing $2\pi$-periodic real analytic function $A(u)$.
 \end{lema}
We need to observe here that the uniqueness stated by Lemma \ref{cano} takes into account the identification that we referred to in the introduction, that is, that two solutions to \eqref{maineq} that overlap on a punctured neighborhood of the isolated singularity are considered to be equal.
   
\begin{definicion}\label{ccp}
We call $\psi(u,v):\Gamma_R\flecha \L^3$ in Lemma \ref{cano} the \emph{canonical conformal parametrization} of the graph of the solution $z\in C^{\omega}(\Omega)$ to \eqref{maineq}.
\end{definicion}
   
 \section{Existence of conelike singularities with prescribed mean curvature}\label{sec:exis}
  
In this section we prove an existence theorem for solutions to \eqref{maineq} with a non-removable isolated singularity at the origin, by prescribing the limit null curve of the solution at the singularity with respect to its canonical conformal structure.

 \begin{teorema}\label{existence}
Let $\cH\in C^{\omega}(\cO)$, $\cH>0$, and let $A:\R\rightarrow\r\backslash\{0\}$ be a $2\pi$-periodic, real analytic  function. Then, there exists a real analytic solution $z\in C^{\omega}(\Omega)\cup C^0(\overline{\Omega})$ to \eqref{maineq} defined on some punctured disk $\Omega$ around the origin, such that:
 \begin{enumerate}
 \item 
The origin is a non-removable isolated singularity of $z$.
 \item
The limit null curve of $z=z(x,y)$ at the origin with respect to its canonical conformal structure is given by $b(u)=A(u)(\cos u,-\sin u,1)$, $u\in\R$.
\item
The Hessian determinant $z_{xx}z_{yy}-z_{xy}^2$ does not vanish around the origin.
  \end{enumerate}
  \end{teorema}
  \begin{proof}
  Let $\psi(u,v)$ be the unique real analytic solution to the Cauchy problem 
  \begin{equation}\label{cauchy}
  \left\{ \def\arraystretch{1.2} \begin{array}{lll} \Delta \psi & = & 2 \cH(\psi) \psi_u \times \psi_v, \\ \psi(u,0) & = & 0, \\ \psi_v (u,0) & = & b(u),\end{array}\right.
  \end{equation}
where $\cH$ and $b(u)$ are given as in the statement of Theorem \ref{existence}. By uniqueness, such a solution $\psi$ is $2\pi$-periodic, and hence it can be defined in a horizontal quotient strip $\hat{\Gamma_R}:=\{(u,v): -R<v<R\}/(2\pi \Z)$, for $R>0$ small enough. We will keep the usual notation for $\Gamma_R$, i.e. $\Gamma_R =\widehat{\Gamma_R} \cap \{v>0\}$.
   
The proof of the theorem will follow by proving the following claims.
  \begin{claim}\label{claim1}
 The map $\psi(u,v):\widehat{\Gamma_R}\flecha \L^3$   satisfies the conformallity conditions 
  \begin{equation}\label{coco}
 \esiz \psi_u,\psi_u\esde = \esiz \psi_v,\psi_v\esde, \hspace{1cm} \esiz \psi_u,\psi_v\esde =0.
  \end{equation}
  Moreover, around points where $\esiz \psi_u,\psi_u\esde > 0$ the map $\psi$ defines a conformally immersed spacelike surface in $\L^3$ whose mean curvature $H$ is given by $H(u,v)=\cH(\psi(u,v))$.
  \end{claim}
  \begin{proof}[Proof of Claim \ref{claim1}]
 The conformal equations \eqref{coco} are equivalent to the complex equation $\esiz \psi_w,\psi_w\esde =0$, where $w=u+iv$. Note that by the initial conditions imposed in \eqref{cauchy}, we have $\esiz \psi_w,\psi_w\esde (u,0)=0$ for every $u$. Moreover, from the PDE in \eqref{cauchy} we see that $\esiz \psi_{w\overline{w}},\psi_w\esde =0$, i.e. $\langle \psi_w,\psi_w\rangle$ is holomorphic. Since this function vanishes along the real axis, we deduce that $\langle \psi_w,\psi_w\rangle(w)=0$ globally on $\widehat{\Gamma_R}$, as wished.

Finally, around points where $\esiz \psi_u,\psi_u\esde >0$ it is clear that $\psi$ is a spacelike conformal immersion, which therefore satisfies the system \eqref{lapinm} with respect to its mean curvature function $H=H(u,v)$. Comparing \eqref{cauchy} with \eqref{lapinm} we conclude that  $H(u,v)=\cH(\psi(u,v))$. This proves Claim \ref{claim1}.  
  \end{proof}
  
  \begin{claim}\label{inm}
Choosing a smaller $R>0$ if necessary, $\psi:\Gamma_R\flecha \L^3$ is a spacelike immersion. \end{claim}
  
 \begin{proof}[Proof of Claim \ref{inm}] 

By \eqref{cauchy}, and assuming that $b(u)$ is as in the statement of the theorem,  we have
 $$(\psi_u\times \psi_v)_v(u,0)=\psi_{uv}\times\psi_v(u,0)=b'(u)\times b(u)= A(u)b(u)\neq {\bf 0},\quad\forall u\in\R.$$  
So, denoting $(n_1,n_2,n_3):=\psi_u\times\psi_v$  we have $n_3(u,0)=0$  and $(n_3)_v(u,0)=A(u)^2>0$. Then, if we write $\psi=(x,y,z)$,  we deduce that  $x_uy_v-x_vy_u=n_3 >0$ in some $\Gamma_{R'}$ with $R'\in (0,R)$. This proves that $\psi:\Gamma_{R'}\flecha \L^3$ is a local graph in the vertical direction, and in particular an immersion. Finally, that the immersion $\psi:\Gamma_{R'}\flecha \L^3$ is spacelike is easily deduced from the conformal conditions \eqref{coco}. This proves Claim \ref{inm}.
 \end{proof}
   
  \begin{claim}\label{emb}
Choosing a smaller $R>0$ if necessary, $\psi(\Gamma_R)$ is a spacelike graph $z=z(x,y)$ in $\L^3$, defined on some punctured disk $\Omega\subset \R^2$ around the origin.  
  \end{claim}
 \begin{proof}[Proof of Claim \ref{emb}]
Since $\psi(u,v)$ is a local graph then the map $\Phi^{-1}:\Gamma_R\fl \r^2$ given by $\Phi^{-1}(u,v)=(x(u,v),y(u,v))$ is a local diffeomorphism. Moreover, since $\Phi^{-1}(u,0)=(0,0)$ and 
$$
\lim_{v\rightarrow 0}\frac{z(u_0,v)^2}{x(u_0,v)^2+y(u_0,v)^2}=\frac{z_v(u_0,0)^2}{x_v(u_0,0)^2+y_v(u_0,0)^2}=1,
$$
we obtain that the immersion $\psi(u,v)$ is tangent to the one of the null cones. A standard topological argument asserts that  $\Phi^{-1}$ is a covering map from an adequate open set $U\subset\Gamma_R$, with $\Gamma_{R'}\subset U$ for $R'$ small enough, onto a punctured neighborhood of the origin. So, in order to prove Claim \ref{emb} we need to show that the number of sheets of this covering is one.

From the expression of $b(u)$ and (\ref{omegag}) we obtain that $g(u,v)$ is real analytic up to the real axis and $g(u,0)=\cos u +i\sin u$. Hence, for every closed curve $\gamma\subset\Gamma_R\cup\r$ homotopic to the circle given by the real axis satisfies that the degree of $(g/|g|)(\gamma)$ is one.     

Since $\psi(u,v)$ is tangent to one of the null cones, we can assume, for instance, that it is tangent to $\n^2_+$. Thus, observe that the number of sheets of the covering map $\Phi^{-1}$ is equal to the degree of the planar curve $\gamma_\varepsilon$ given by $\psi(u,v)\cap\{z=\varepsilon\}$, for $\varepsilon>0$ small enough. Again, as explained in the proof of Lemma \ref{univ}, the degree of this planar curve can be computed as the degree of $(\mu/|\mu|)(\gamma_\varepsilon)=(g/|g|)(\gamma_{\varepsilon})$. Therefore, $\Phi^{-1}$ is a diffeomorphism from $U$ onto its image as we wanted to show.
\end{proof}

With all of this, we have obtained a function
$z\in C^{\omega}(\Omega)\cup C^0(\overline{\Omega})$ such that its graph $z=z(x,y)$ in $\L^3$ can be conformally parametrized by the map $\psi(u,v):\Gamma_R\flecha \L^3$ that solves \eqref{cauchy}. As proved in Claim \ref{claim1}, the graph has prescribed mean curvature given by $\cH$; thus, $z$ is an elliptic solution to \eqref{maineq}, with an isolated singularity at the origin, and $z(0,0)=0$ (since $\psi(u,0)=(0,0,0)$). This singularity is non-removable since the induced conformal structure is that of an annulus.  From Remark \ref{rem:K}    $z_{xx}z_{yy}-z_{xy}^2$ does not vanish around the origin since the Gaussian curvature of the graph blows up at the origin.

 Moreover, again by the initial conditions in \eqref{cauchy} and the specific expression of $b(u)$, we see that $\psi(u,v)$ is indeed the canonical conformal parametrization of the graph $z=z(x,y)$ (see Definition \ref{ccp}), and that $b(u)$ is its limit null curve for this parametrization. This finishes the proof of Theorem \ref{existence}.
  \end{proof}

 \section{Classification Theorem and examples}\label{sec:class}
 
 We next use our results of the previous sections to deduce the classification result for non-removable isolated singularities of elliptic solutions to \eqref{maineq} stated in Theorem \ref{main1}.  Theorem \ref{main2} below gives a more specific statement for such theorem.

 \begin{teorema}\label{main2}
  Let $\cO\subset \L^3$ be a neighborhood of some point $p_0=(x_0,y_0,z_0)$, and let $\cH\in C^{\omega}(\cO)$, $\cH>0$.
 
 Let $\cA_1$ denote the class of all elliptic solutions $z(x,y)$ to \eqref{maineq} that satisfy the following conditions:
  \begin{enumerate}
  \item
$z\in C^{2}(\Omega)$, where $\Omega\subset \R^2$ is some punctured disk (of any radius) centered at $(x_0,y_0)$.
 \item
$z\in C^0(\overline{\Omega})$, with $z(x_0,y_0)= z_0$.
 \item
The Hessian determinant $z_{xx} z_{yy} -z_{xy}^2$ does not vanish around $(x_0,y_0)$.
 \end{enumerate}
Here, we identify two elements of $\cA_1$ if they coincide on a neighborhood of $(x_0,y_0)$. 

Let $\cA_2$ denote the class of $2\pi$-periodic, real analytic, nowhere vanishing functions $A(u):\R\flecha \R\backslash \{0\}$.

Then, the map that sends each $z\in \cA_1$ to the height function $A(u)$ of its limit null curve at the singularity $p_0$ with respect to its canonical conformal parametrization defines a bijective correspondence between $\cA_1$ and $\cA_2$.
\end{teorema}

\begin{proof}

Without loss of generality we can assume $p_0=(0,0,0)$.

Consider the map $\Upsilon$ that sends each $z\in
\cA_1$ to the height function $A(u)$ of its limit null curve with respect to its canonical conformal parametrization. By Lemma \ref{cano}, $\Upsilon$ is a well defined map from $\cA_1$ into $\cA_2$. So, in order to prove
Theorem \ref{main2} it remains to check that $\Upsilon:\cA_1\flecha \cA_2$ is
bijective.

Surjectivity is a consequence of Theorem \ref{existence}, as
follows. Consider $A(u)\in \cA_2$, and construct the curve $$b(u)=A(u)(\cos u,-\sin u,1):\R/(2\pi \Z)\flecha \N^2_+\cup \N_-^2.$$ 
By Theorem \ref{existence}, there exists an elliptic solution $z\in C^{\omega}(\Omega)\cup C^0(\overline{\Omega})$ to \eqref{maineq}
with $z(0,0)=0$ satisfying (1), (2) and (3), that is, $z\in \cA_1$. Note that by construction, the map $\Upsilon: \cA_1\flecha \cA_2$ defined above takes this function $z\in \cA_1$ to the function $A(u)\in \cA_2$ we started with.

To finish we prove the injectivity of $\Upsilon$. Let $z_1,z_2\in
\cA_1$ and  let $\psi_1(u,v),\psi_2(u,v):\Gamma_R\flecha \L^3$ denote their respective canonical conformal parametrizations. If $\Upsilon(z_1)=\Upsilon(z_2)=A(u)\in \cA_2$ then both $\psi_1,\psi_2$ are solutions to the Cauchy problem \eqref{cauchy} with the same analytic initial conditions. By
uniqueness of the solution to the Cauchy problem  \eqref{cauchy},
we get $\psi_1 (u,v)= \psi_2(u,v)$. In particular, $z_1=z_2$ on a neighborhood of the origin. This proves injectivity and finishes the proof of
Theorem \ref{main2}.
\end{proof}

We remark that, geometrically, Theorem \ref{main2} can be restated as follows.
 \begin{teorema}\label{class}
 Let $ \mathcal{H}:\mathcal{O}\subset\L^3\flecha (0,\8)$ be a positive
real analytic function defined on an open  set $\mathcal{O}\subset
\L^3$ containing a given point $p_0\in\L^3$. Let $\cA$ denote the
class of all spacelike graphs $\Sigma$ in $\L^3$ with upwards-pointing unit normal
that have $p_0$ as a non-removable isolated singularity, whose
mean curvature at every point $a\in \Sigma\cap \mathcal{O}$ is
given by $\mathcal{H}(a)$, and whose Gaussian curvature does not vanish near $p_0$; here, we identify $\Sigma_1,\Sigma_2\in \cA$
if they overlap on an open set containing the singularity $p_0$.

Then, the map that sends each graph in $\cA$ to its limit null curve at the singularity with respect to its canonical conformal parametrization provides a one-to-one correspondence
between $\cA$ and the class $\cB$ of  regular, spacelike, negatively oriented  Jordan curves in the union of the positive and negative null cones $\N_+^2\cup \N_-^2$.
 \end{teorema}
   
We say that $\cH\in C^{\omega}(\cO)$ is \emph{rotationally symmetric} around the point $p_0$ if $\cH\circ I_{\theta} = \cH$ for all rotations $I_{\theta}:\L^3\flecha \L^3$, $\theta\in [0,2\pi)$, with axis $L$ equal to the vertical line of $\L^3$ passing through $p_0$. In this situation, we have:

 \begin{corolario}\label{radial}
Assume that  $\cH\in C^{\omega}(\cO)$ is \emph{rotationally symmetric} around the point $p_0$, and let $\Upsilon:\cA_1\flecha \cA_2$ be the bijective correspondence given by Theorem \ref{main2}. Then, the subclass of radial graphs in $\cA_1$ is mapped via $\Upsilon$ to the subclass of constant, non-zero functions in $\cA_2$.
\end{corolario}
\begin{proof}
If $A(u)\in \cA_2$ is constant and $\cH$ is rotationally symmetric, then the solution $\psi(u,v)$ to the Cauchy problem \eqref{cauchy} is rotationally symmetric, since by uniqueness of the solution and rotational symmetry of the initial data it satisfies 
 \begin{equation}\label{eqro}
I_{\theta}( \psi(u,v)) = \psi(u+\theta,v)\end{equation} for every $\theta\in [0,2\pi)$. And conversely, it is easy to check that if $\cH$ is rotationally symmetric and $z\in \cA_1$ is an elliptic radial solution to \eqref{maineq}, then its canonical conformal parametrization $\psi(u,v)$ satisfies \eqref{eqro}, and thus $\Upsilon(z) \in \cA_2$ is a constant function. This easily implies by Theorem \ref{main2} that a graph $z\in \cA_1$ with a non-removable isolated singularity will be radial if and only if $A(u)=\Upsilon(z)\in \cA_2$ is constant.
\end{proof}
\begin{ejemplo*}
Let us make the choices $\cH=1$ and $A(u)= \pm\frac{1}{4}\in \cA_2$. In both cases, the graph $z\in \cA_1$ given by $z=\Upsilon^{-1}(A(u))$ in Theorem \ref{main2} is conformally parametrized by the unique solution to the Cauchy problem \eqref{cauchy}.
 
By Corollary \ref{radial}, this gives rise to  radially symmetric elliptic solutions to \eqref{maineq}, and thus to  rotationally invariant spacelike surfaces of constant mean curvature $H=1$ in $\L^3$.

We can parametrize such  solutions $\psi(u,v)$ to \eqref{cauchy} as
$$\psi(u,v)= (\cos(u) f(v),-\sin(u)f(v),h(v)),$$
where, in the case $A(u)= -\frac{1}{4}$, $f(v)=-\frac{1}{2}\tan(v/2)$ and $h(v)=-\frac{1}{2}(v-\tan(v/2))$.
In the case $A(u)=  \frac{1}{4}$,   the functions $f(v)$ and $h(v)$ are the unique solutions to the following Cauchy problem:
$$\left\{\begin{array}{lcl}
f'(v)& =& (1/16 + 3/2 f(v)^2 + f(v)^4)^{1/2},\\
 h'(v)& = &1/4 + f(v)^2, \\
 f(0)&=&0,\\ h(0)&=&0.\end{array}\right.$$

In Figure \ref{f1} we   show the surfaces obtained above.
\vspace{-0.5cm}
\begin{figure}[H]
\centering
\subfloat[Case $A(u)=-1/4$.]{
    \includegraphics[width=6cm]{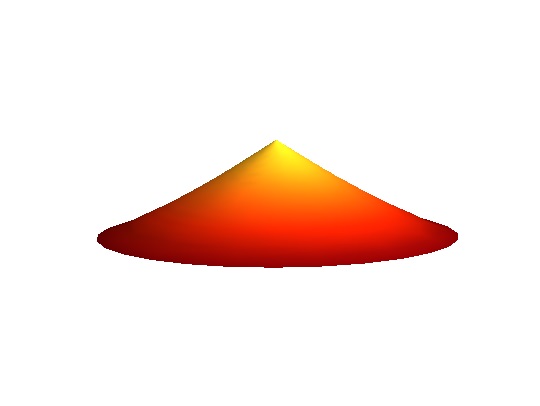}}
  \hfill
  \subfloat[Case $A(u)=1/4$.]{
    \includegraphics[width=6cm]{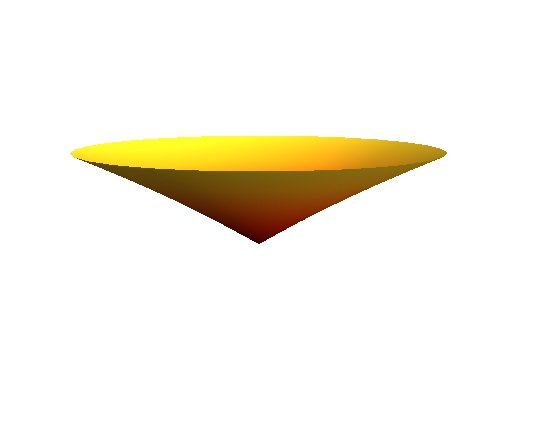} }
  \caption{Radial spacelike graphs with constant mean curvature $H=1$ and a conelike singularity.}
  \label{f1}
   \end{figure}
   
\end{ejemplo*}

 \end{document}